\documentclass[12pt]{amsart}

\usepackage{amssymb,amsthm,amsmath}

\usepackage{mathtools}

\RequirePackage[dvipsnames,usenames]{color}
\input{kmacros3.sty}
\usepackage{mabliautoref}

\usepackage{tikz}
\usepackage{tikz-cd}
\usetikzlibrary{cd}
\usepackage{graphicx}
\usepackage[all,cmtip]{xy}

\numberwithin{equation}{theorem}

\usepackage{bm}
\usepackage{pifont}
\usepackage{upgreek}

\usepackage{eucal}
\usepackage{ulem}
\usepackage{stmaryrd}

\numberwithin{equation}{theorem}

\usepackage{fullpage}

\usepackage{enumerate}
\usepackage{calc}

\usepackage{verbatim}
\usepackage{alltt}
\usepackage{scalerel,stackengine}

\newcommand{\perf}{\textnormal{perf}}

\stackMath
\newcommand\reallywidehat[1]{%
\savestack{\tmpbox}{\stretchto{%
  \scaleto{%
    \scalerel*[\widthof{\ensuremath{#1}}]{\kern.1pt\mathchar"0362\kern.1pt}%
    {\rule{0ex}{\textheight}}
  }{\textheight}%
}{2.4ex}}%
\stackon[-6.9pt]{#1}{\tmpbox}%
}
\begin{document}

\title{BCM-thresholds of hypersurfaces}
\author[S. Rodr\'iguez-Villalobos]{Sandra Rodr\'iguez-Villalobos}
\address{Department of Mathematics, University of Utah, 155 South 1400 East, 
Salt Lake City, UT 84112, USA}
\email{rodriguez@math.utah.edu}
\thanks{The author was partially supported by NSF grant DMS-2101800}
\maketitle

\begin{abstract}
	In this paper, we use big Cohen-Macaulay algebras to define a characteristic free analog of the $F$-thresholds, which we call BCM-thresholds, in the case of principal ideals.
	We prove that, similarly to the case of the $F$-thresholds, the set of BCM-thresholds and the set of BCM-jumping numbers agree.
	We also relate some BCM-thresholds to splittings of maps from the ring to a big Cohen-Macaulay algebra.
\end{abstract}
\section{Introduction}

Given a Noetherian ring $R$ of prime characteristic $p$, $J$ an ideal of $R$ and $f\in\sqrt{J}$, the $F$-threshold of $f$ with respect to $J$ is given by 
$$c^J(f)=\lim_{e\to\infty}\frac{\nu_e^J(f)}{p^e},$$
where $\nu_e^J(f)=\max\{n\mid f^n\not \in J^{[p^e]}\}$. 
These invariants were introduced by Musta\c ta, Takagi and Watanabe in the case of an $F$-finite regular local ring \cite{MustataTakagiWatanabeFThresholdsAndBernsteinSato} and De Stefani, Nu\~nez-Betancourt and Perez proved the existence of their defining limit for all $F$-finite rings (\cite{DeStefaniNunezBetancourtPerezFThresholdsAndRelated}).
The $F$-thresholds are known to be related to tight closure, integral closure, multiplicities and $a$-invariants (\cite{HunekeMustataTakagiWatanabeFThresholdsTightClosureIntClosureMultBounds, MR3778235,DeStefaniNunezBetancourtPerezFThresholdsAndRelated}).
When $R$ is regular and $F$-finite, the set of $F$-thresholds agrees with the set of $F$-jumping numbers for the test ideal of $f$ (\cite{MustataTakagiWatanabeFThresholdsAndBernsteinSato, BlickleMustataSmithDiscretenessAndRationalityOfFThresholds}).  

In this paper, we express the $F$-threshold using $R^+$ for strongly-$F$-regular rings.
We use this result to define a characteristic free analog of the $F$-thresholds, which we call BCM-thresholds.
\begin{definition}[see \autoref{BCMThresholdDef}]
	Let $R$ be a Noetherian ring, not necessarily of prime characteristic, and $B$ a big Cohen-Macaulay $R^+$-algebra.
	We define 
	$$c^J_B(f)=\inf\left\{t\in\R_{\geq0}\middle| (fR^+)^{>t}\sub JB \right\}=
	\sup\left\{t\in\R_{\geq0}\middle| (fR^+)^{>t}\not\sub JB \right\}$$
	where $(fR^+)^{>t}=\{g\in R^+\mid g^n\in f^m R^+ \text{ for some } \frac{m}{n}>t\}.$
\end{definition}

In our main result, we relate the set of BCM-thresholds and the set of BCM-jumping numbers for the BCM-test ideal of $f$.

\begin{theorem}[see \autoref{Main}]
	Let $f$ be an element of a complete regular local ring $R$ and $B$ a big Cohen-Macaulay $R^+$-algebra. The set of $B$-jumping numbers for $f$ is the same as the set of $B$-thresholds of $f$.
\end{theorem}

In future work, I plan to explore the case of non-principal ideals.

\textbf{Acknowledgements:} The author would like to thank her advisor Karl Schwede for suggesting this problem to me, for many valuable discussions, and for his support and guidance.
The author would also like to thank Rankeya Datta for pointing out that one of the properties we were missing was proved in \cite{datta2023mittagleffler}.
Additionally, the author would like to thank Jack Jeffries and Ilya Smirnov for helpful comments.

\section{$F$-thresholds via $R^+$}

We start by rewriting the $F$-threshold in terms of $R_{\perf}$.

\begin{proposition}\label{propPerfThreshold}
	Let $R$ be a Noetherian $F$-finite $F$-pure ring, $J$ an ideal of $R$ and $f\in\sqrt{J}$. 
	$$c^J(f) = \inf\left\{\frac{a}{p^e}\middle| f^{\frac{a}{p^e}}\in JR_{\mathrm{perf}}\right\}=
	\sup\left\{\frac{a}{p^e}\middle| f^{\frac{a}{p^e}}\not\in JR_{\mathrm{perf}}\right\}.$$
\end{proposition}
\begin{proof}
	Suppose that $f^{a/p^e}\in JR_{\mathrm{perf}}$.
	It follows that  $f^{a/p^e}\in JR^{1/p^{e'}}$ for some $e'\geq e$.
	Given that $R$ is $F$-pure, there exists $\psi\in \Hom_{R^{1/p^{e}}}(R^{1/p^{e'}},R^{1/p^{e}})$ with $\psi(1)=1$, so
	\begin{align*}
	f^{a/p^e}&=f^{a/p^e}\psi(1)\\
	&=\psi\left(f^{a/p^e}\right)\in \psi\left(JR^{1/p^{e'}}\right)=JR^{1/p^{e}}.
	\end{align*}
	Thus, $f^a\in J^{[p^{e}]}$.
	Hence, we have that $ f^{ap^{e'-e}} \in J^{[p^{e'}]}$ for all $e'\geq e.$
	As a consequence, 
	$$\frac{\nu_{e'}^J(I)}{p^{e'}}\leq \frac{ap^{e'-e}}{p^{e'}}=\frac{a}{p^{e}}.$$
	Therefore,
	$$c^{J}(f)\leq\inf\left\{\frac{a}{p^e}\middle| f^{\frac{a}{p^e}}\in JR_{\mathrm{perf}}\right\}.$$
	
	Now, suppose that $f^{a/p^e}\not\in JR_{\mathrm{perf}}$.
	Then, $f^{a/p^e}\not\in JR^{1/p^{e'}}$ for all $e'$ and,
	as a consequence, $f^{ap^{e'-e}}\not\in J^{[p^{e'}]}$ for all $e'\gg e$.
	It follows that
	$$\frac{a}{p^{e}}=\frac{ap^{e'-e}}{p^{e'}}\leq\frac{\nu_{e'}^J(f)}{p^{e'}}.$$
	Therefore,
	$$c^{J}(f)\geq\sup\left\{\frac{a}{p^e}\middle| f^{\frac{a}{p^e}}\sub JR_{\mathrm{perf}}\right\}.$$
\end{proof}

If $R$ is strongly $F$-regular, we obtain a similar result when considering small perturbations.

\begin{proposition}\label{propPerfTightThreshold}
	Let $R$ be a Noetherian $F$-finite strongly $F$-regular ring, $J$ an ideal of $R$ and $f\in\sqrt{J}$.
	Then,
	$$
	c^J(f) = \inf\left\{\frac{a}{p^e}\middle| d^{\frac{1}{p^{e'}}}f^\frac{a}{p^e}\in JR_{\mathrm{perf}} \text{ for some } d\in R \text{ nonzerodivisor and all } e'\gg 0  \right\}.
	$$	
\end{proposition}
\begin{proof}
	It follows from \autoref{propPerfThreshold} that
	$$
	c^J(f) \geq\inf\left\{\frac{a}{p^e}\middle| d^{\frac{1}{p^{e'}}}f^\frac{a}{p^e}\in JR_{\mathrm{perf}} \text{ for some } d\in R \text{ nonzerodivisor and all } e'\gg 0  \right\}.
	$$
	
	Suppose that $d^{\frac{1}{p^{e'}}}f^\frac{a}{p^e}\in JR_{\mathrm{perf}}$ for some nonzerodivisor $d\in R$ and all $e'\gg  e$.
	We have that $d^{\frac{1}{p^{e'}}}f^\frac{a}{p^e}\in JR^{1/p^{e''}}$ for some $e''>e'$.
	Given that $R$ is $F$-split, there exists $\psi\in \Hom_{R^{1/p^{e'}}}(R^{1/p^{e''}},R^{1/p^{e'}})$ with $\psi(1)=1$.
	Thus, $$d^{\frac{1}{p^{e'}}}f^\frac{a}{p^e}=\psi\left(d^{\frac{1}{p^{e'}}}f^\frac{a}{p^e}\right)\in JR^{1/p^{e'}}$$ for all $e'>0$.
	Since $R$ is strongly $F$-regular, there exists an $R^{1/p^s}$-linear map $\varphi: R^{1/p^{i}}\rightarrow R^{1/p^s}$ with $\varphi(d^{1/p^{i}})=1$ for some $s,i>e$.
	We have that
	\begin{align*}
		\varphi\left(d^{\frac{1}{p^i}}f^\frac{a}{p^e}\right)&= f^\frac{a}{p^e} \varphi\left(d^{\frac{1}{p^{i}}}\right)\\
		&= f^\frac{a}{p^e} 
	\end{align*}
	Thus, $f^\frac{a}{p^e} \in\varphi(JR^{1/p^{i}})\subseteq JR_{\mathrm{perf}}$.
	By \autoref{propPerfThreshold},
	$$
	c^J(x) \leq \inf\left\{\frac{a}{p^e}\middle| d^{\frac{1}{p^{e'}}}f^\frac{a}{p^e} \in JR_{\mathrm{perf}} \text{ for some } d\in R \text{ nonzerodivisor and all } e'\gg 0  \right\}.
	$$
\end{proof}

We will use the following well-known fact (see, for example \cite{SmithTightClosureParameter}) to express the $F$-threshold using $R^+$:

\begin{lemma}\label{lemmaTight}
	Let $R\subseteq S$ be a finite extension of Noetherian domains. 
	Then for every ideal $I\subseteq R$ there exists a nonzero $x\in R$ such that 
	$$x(IS\cap R)^{[p^e]}\subseteq I^{[p^e]}$$
	for all $e\in\mathbb{N}.$
\end{lemma}
\begin{proof}
	Given that $\mathrm{frac}(R)\otimes \Hom_R(S,R)=\Hom_{\mathrm{frac}(R)}(\mathrm{frac}(S),\mathrm{frac}(R))\neq 0$, we have that there exists a map $\varphi\in\Hom_R(S,R)$ with $\varphi(1)\neq 0$.
	Let $f\in IS\cap R$. Then $f=\sum_{i=1}^na_ig_i$ where $g_i\in I$ and $a_i\in S$ for all $i$.
	Thus, $f^{p^e}=\sum_{i=1}^na_i^{p^e}g_i^{p^e}$ and, applying $\varphi$, 
	\begin{align*}
		\varphi(1)f^{p^e}=\varphi(f^{p^e})=\sum_{i=1}^n\varphi(a_i^{p^e})g_i^{p^e}\in I^{[p^e]}.
	\end{align*}
	Since $\varphi(1)\neq 0$ and $f\in IS\cap R$ is arbitrary, the result follows.
\end{proof}

\begin{proposition}\label{R+Threshold}
	Let $R$ be a Noetherian $F$-finite strongly $F$-regular domain, $J$ an ideal of $R$ and $f\subseteq\sqrt{J}$. Then,
	$$
	c^J(f) = \inf\left\{\frac{a}{p^e}\middle| f^{\frac{a}{p^e}}\in JR^{+} \right\}=
	\sup\left\{\frac{a}{p^e}\middle| f^{\frac{a}{p^e}}\not\in JR^{+} \right\}.
	$$
\end{proposition}
\begin{proof}
	It follows from \autoref{propPerfThreshold} that
	$$
	c^J(f) \geq \inf\left\{\frac{a}{p^e}\middle| f^{\frac{a}{p^e}}\in JR^+ \right\}.
	$$
	
	Suppose that $f^{\frac{a}{p^e}}\in JR^+$. 
	Then there exists a finite extension $R^{1/p^e}\subseteq S$ such that $f^{\frac{a}{p^e}}\in JS$.
	By \autoref{lemmaTight},  there exists $x\in R^{1/p^e}\setminus \{0\}$ such that 
	$$x(JS\cap R^{1/p^e})^{[p^s]}\subseteq (JR^{1/p^e})^{[p^s]}$$
	for all $s$.
	Since
	$$x^{p^{e}}(JS\cap R^{1/p^e})^{[p^s]}\subseteq (JR^{1/p^e})^{[p^s]}$$
	for all $s$, we may assume that $x\in R$.
	It follows that
	$$x^{\frac{1}{p^s}}f^{\frac{a}{p^e}}\in x^{\frac{1}{p^s}}(JS\cap R^{1/p^e})R_{\mathrm{\perf}}\subseteq JR_{\mathrm{\perf}}.$$
	for all $s$.
	By \autoref{propPerfTightThreshold}, we have that $\frac{a}{p^e}\geq c^J(f)$. Thus,
	$$
	c^J(f) \leq \inf\left\{\frac{a}{p^e}\middle| f^{\frac{a}{p^e}}\in JR^+ \right\}.
	$$
\end{proof}

We would like to characterize $F$-thresholds without using the Frobenius map. We begin by defining a notion of real powers of a principal ideal for a ring extension $R\hookrightarrow S$.
\begin{definition}[\textit{cf.}\cite{HunekeSwansonIntegralClosure,bisui2024rational}]
	Let $R$ be a Noetherian ring, not necessarily of prime characteristic, $S$ and $R$-algebra, $f\in R$ and $t\in \R_{\geq 0}$. We define 
	$$(fS)^{>t}=\{g\in S\mid g^n\in f^m S \text{ for some } \frac{m}{n}>t\}.$$
\end{definition}

\begin{remark}
	Suppose that $g^n\in f^m S$ for some $\frac{m}{n}>t$. Let $e$ be such that
	$$\frac{m}{n}-\frac{m}{p^e}>t$$ and let $a$ be such that $p^e>an$ and $p^e-an<n$.
	Since $g^{an}\in f^{am} S$, we have that $g^{p^e}\in f^{am} S$.
	On the other hand, 
	\begin{align*}
		t<\frac{m}{n}-\frac{m}{p^e}=\frac{(p^e-n)m}{np^e}\leq \frac{anm}{np^e}=\frac{am}{p^e}.
	\end{align*}
	It follows that
	$$(fS)^{>t}=\{g\in S\mid g^{p^e}\in f^a S \text{ for some } \frac{a}{p^e}>t\}$$
	We also have that
	$$(fS)^{>t_2}\sub (fS)^{>t_1}$$
	if $t_1<t_2$
	and
	$$(fS)^{>t}=\bigcup_{s>t} (fS)^{>s}.$$
\end{remark}

When the ring is of prime characteristic and $S=R^+$, it suffices to consider the ideal of $S^+$ generated by elements of the form $f^{a/p^e}$, where $a/p^e>t$.
\begin{lemma} \label{lemmaDescriptionOfPowersInPrimeCharacteristic}
	If $R$ is a Noetherian ring of characteristic $p>0$ and $f\in R$,
	$$(fR^+)^{>t}=\{f^{a/p^e}\, \mid \, a/p^e>t\}R^+.$$
\end{lemma}
\begin{proof}
	We have that
	$\{f^{a/p^e}\, \mid \, a/p^e>t\}R^+\sub (fR^+)^{>t}$.
	Now suppose that $g\in (fR^+)^{>t}$.
	 Then, $g^{p^e}\in f^aR^+$ for some $a/p^e>t$ and $g\in f^{a/p^e}R^+$.	
\end{proof}

The following characterization now follows immediately.
\begin{corollary}
	Let $R$ be a Noetherian $F$-finite strongly $F$-regular domain, $J$ an ideal of $R$ and $f\subseteq\sqrt{J}$. Then,
	$$
	c^J(f) = \inf\left\{t\in\R_{\geq0}\middle| (fR^+)^{>t}\sub JR^{+} \right\}=
	\sup\left\{t\in\R_{\geq0}\middle| (fR^+)^{>t}\not\sub JR^{+} \right\}.
	$$
\end{corollary}

\section{BCM-thresholds}

Hochster and Huneke proved the existence of a big Cohen-Macaulay algebra in equal characteristic.
In particular, when $R$ is a local Noetherian domain of prime characteristic, Hochster and Huneke (\cite{HochsterHunekeInfiniteIntegralExtensionsAndBigCM}) proved that $R^+$ is a big Cohen-Macaulay $R$-algebra.  
In mixed characteristic, Andr\'e (\cite{AndreWeaklyFunctorialBigCM}) proved the existence of big Cohen-Macaulay algebras,
and Bhatt (\cite{BhattAbsoluteIntegralClosure}) proved that the $p$-completion of $R^+$ is a big Cohen-Macaulay algebra.

Inspired by the previous section, we define a characteristic free analog of the $F$-thresholds.

\begin{definition}[BCM-thresholds]\label{BCMThresholdDef}
	Let $R$ be a Noetherian ring, not necessarily of prime characteristic, and $B$ a big Cohen-Macaulay $R^+$-algebra.
	We define 
	$$c^J_B(f)=\inf\left\{t\in\R_{\geq0}\middle| (fR^+)^{>t}\sub JB \right\}=
	\sup\left\{t\in\R_{\geq0}\middle| (fR^+)^{>t}\not\sub JB \right\}.$$
\end{definition}

We now record some basic properties of this thresholds.
\begin{proposition}
	Let $R$ be a Noetherian ring, not necessarily of prime characteristic, $I,J$ be ideals of $R$, $f\in \sqrt{J}$, and $B$ a big Cohen-Macaulay $R^+$-algebra.
	\begin{enumerate}
		\item If $J\subseteq I$, then $c^I_B(f)\leq c^J_B(f)$
		\item $c^J_B(f^r)=\frac{1}{r}c^J_B(f)$ for $r\in\Z_{>0}$.
		\item If $R\hookrightarrow S$ is an integral extension of complete domains, then
		$$c^J_B(f)=c^{JS}_B(f)$$
	\end{enumerate}
\end{proposition}
\begin{proof}$ $
	\begin{enumerate}
		\item It follows directly from the definition.
		\item We have that
			\begin{align*}
				(f^rR^+)^{>t}&=\{g\in R^+\mid g^n\in f^{rm} R^+ \text{ for some } \frac{m}{n}>t\}\\
					&=\{g\in R^+\mid g^n\in f^m R^+ \text{ for some } \frac{m}{n}>rt\}\\
					&=(fR^+)^{>rt}.
			\end{align*}
			Thus, 
			\begin{align*}
				c^J_B(f^r)=\inf\left\{t\in\R_{\geq0}\middle| (fR^+)^{>rt}\sub JB \right\}=\frac{1}{r}c^J_B(f)
			\end{align*}
		\item Since $S^+=R^+$, we have that $(fS^+)^{>t}=(fR^+)^{>t}$. Thus, $c^J_B(f)=c^{JS}_B(f)$.
	\end{enumerate}
\end{proof}

We want to turn our attention to an analog of an $F$-pure thresholds. In order to do so, we will now focus on the case where $(R,m)$ is a complete regular local ring. In this case, the ring satisfies the following useful property:

\begin{definition}[\cite{datta2023mittagleffler}]
	We say that a ring homomorphism $R\rightarrow S$ is Ohm-Rush trace (ORT) if, 
	for every $x\in S$,
	$$x\in \{\varphi(x) |  \varphi\in \Hom_R(S,R)\}S$$
\end{definition}

Under these conditions, $c^m_B(f)$ provides information about when the map $R\rightarrow B$ sending $1 \text{ to } f^t$ splits.
In order to prove this, we will use the following notation.

\begin{definition}[\textit{cf.} \cite{AberbachEnescuStructureOfFPure}]
	Let $	J$ be an ideal of $R$, and $B$ a big Cohen-Macaulay $R$-algebra. 
	We define
	$$I_{B}(J)=\{x\in B | \varphi(x)\in J \text{ for all } \varphi\in \Hom_R(B,R)\}.$$
\end{definition}

When $R$ is a complete regular local ring, we have the following result, inspired by Fedder's Criterion (\cite{FedderFPureRat}).
\begin{lemma}[\cite{datta2023mittagleffler}] \label{lemmaEqual}
	Let $J$ be an ideal of a complete regular local ring $R$, and $B$ a big Cohen-Macaulay $R^+$-algebra. 
	We have that $I_{B}(J)=JB$.
\end{lemma}
\begin{proof}
	Let $x\in I_{B}(J)$. 
	By Corollary 5.3.16. in \cite{datta2023mittagleffler}, $B$ is ORT.
	Thus, $$x\in\{\varphi(x) | \varphi\in \Hom_R(B,R)\}B\subseteq JB.$$
	Since $I_{B}(J)\supseteq JB$, it follows that $I_{B}(J)=JB.$
\end{proof}

We now connect $I_{B}(m)$ back to splittings.

\begin{proposition}
	Let $(R, m)$ be a complete regular local ring, $f$ be a nonzero element of R, and $B$ a big Cohen-Macaulay $R^+$-algebra. 
	Then, for every  $t\in \mathbb{Q}$, $f^t\not\in mB$ if and only if the map $R\rightarrow B$ sending $1$ to $f^t$ splits.
\end{proposition}
\begin{proof}
	If $f^t\in mB$, then the map $R\rightarrow B$ sending $1$ to $f^t$ does not split.
	
	If $f^t\not\in mB$, by \autoref{lemmaEqual}, $f^t\not\in I_B(m)$.
	Thus there exists $\varphi\in \Hom_R(B,R)$ such that $\varphi(f^t)\not \in m$.
	It follows that the map $R\rightarrow B$ sending $1$ to $f^t$ splits.
\end{proof}

\begin{definition}
	Let $R$ be a Noetherian domain, not necessarily of prime characteristic, $f\in R$, and $B$ a big Cohen-Macaulay $R^+$-algebra.
	We define 
	$$c_B(f)=\sup\left\{t\in\Q_{\geq0}\middle| \text{the map } R\rightarrow B \text{ sending } 1 \text{ to } f^t \text{ splits}\right\}.$$
\end{definition}

\begin{remark}
	Assume $f\neq0$. Let $b\in\ZZ_{>0}$ and suppose that $g^b=h^b=f$ for $g,h\in R^+$. 
	Then there exists an $b$-th root of unity $u$ such that $g=uh$.
	Thus, if $t=\frac{a}{b}\in\mathbb{Q}_{>0}$, the map $R\rightarrow B$ sending $1$ to $g^a$ splits if and only if
	the map $R\rightarrow B$ sending $1$ to $h^a$ splits.
	It follows that $c_B(f)$ is well-defined.
\end{remark}

\begin{corollary}
	Let $(R, m)$ be a complete regular local ring and $B$ a big Cohen-Macaulay $R^+$-algebra. Then 
	$c^m_B(f^t)=c_B(f^t)$
\end{corollary}


\section{BCM-test ideals}

We now turn our attention to BCM-test ideals (see \cite{MaSchwedeTuckerWaldronWitaszekAdjoint, MaSchwedeSingularitiesMixedCharBCM,MaSchwedePerfectoidTestideal, BMPSTWW1, MurayamaUniformBoundsOnSymbolicPowers,HaconLamarcheSchwede, BMPSTWW2, RobinsonBCMTestIdealsMixedCharToric}). 
\begin{definition}
	Let $R$ be a complete Noetherian ring, not necessarily of prime characteristic, and $B$ a big Cohen-Macaulay $R^+$-algebra.
	For $t>0$,we define 
	$$\tau_{B}(f^{t})=\{\varphi(x) | x\in (fR^+)^{>t}B,\, \varphi\in \Hom(B,R)\}.$$
	Additionally, if $R$ has a canonical module $\omega_R$, we define
	$$\tau_{B}(\omega_R, f^{t})=\{\varphi(x) | x\in (fR^+)^{>t}B,\, \varphi\in \Hom(B,\omega_R)\}.$$
\end{definition}

\begin{remark}
	Since $R$ is Noetherian, 
	for every $t$ there exists $\epsilon>0$ such that $\tau_{B}(f^{t_1})=\tau_B(f^{t_0})$ if $t_1,t_0\in(t,t+\epsilon)$.
	Since $$(fR^+)^{>t}=\bigcap_{s>t} (fR^+)^{>s},$$ it follows that
	$$\tau_{B}(f^{t})=\{\varphi(x) | x\in (fR^+)^{>t}B,\, \varphi\in \Hom(B,R)\}=\tau_B(f^{t_0})$$
	for all $t_0\in[t,t+\epsilon)$.
	Similarly, for every $t$ there exists $\epsilon>0$ such that $\tau_{B}(\omega_R, f^{t})=\tau_B(\omega_R, f^{t_0})$ if $t_0\in[t,t+\epsilon)$.
\end{remark}

\begin{definition}
	We say that $\alpha\in R_{\geq 0}$ is a $B$-jumping number for $f$ if $\tau_{B}(f^{\alpha})\neq\tau_B(f^{\alpha-\epsilon})$ for all $\epsilon$.
\end{definition}

Under mild hypotheses, $\tau_{R^+}(f^{t})$ agrees with $\tau(f^t)$. We expect this is well-known for the experts, but we include a proof for the reader's convenience.

\begin{proposition}
	Let $f$ be an element of an $F$-finite complete local domain $R$ with canonical module $\omega_R$. For $t>0$
		$$\tau(\omega_R, f^t)=\tau_{R^+}(\omega_R, f^{t}).$$
	In particular, if $R$ is quasi-Gorenstein,
		$$\tau(f^t)=\tau_{R^+}(f^{t}).$$
\end{proposition}
\begin{proof}
	We first prove that $\tau(\omega_R, f^t)=\{\varphi(f^t) | \varphi\in \Hom(R^+,\omega_R)\}$ for $t\in\mathbb{Q}$. 
	Let $S$ be a finite extension of $R$ such that $f^t\in S$.
	Then, by Proposition 4.4 in \cite{BlickleSchwedeTuckerTestAlterations},
	$$\tau(\omega_R, f^t)=\Tr(\tau(\omega_S,f^t))=\Tr(f^t\tau(\omega_S)).$$
	On the other hand, we have that
	\begin{align*}
		\Hom_R(R^+,\omega_R)&=\Hom_R(R^+\tensor S,\omega_R)\\
		&=\Hom_S(R^+, \Hom_R(S,\omega_R))\\
		&=\Hom_S(R^+, \omega_S).
	\end{align*}
	Under this correspondence, 
	\begin{align*}
		\{\varphi(f^t) | \varphi\in \Hom(R^+,\omega_R)\}&=\{\psi(f^t)(1) | \psi\in\Hom_S(R^+, \Hom_R(S,\omega_R))\}\\
		&=\Tr(f^t\{\psi(1) | \psi\in\Hom_S(R^+,\omega_S)\}).
	\end{align*}
	Since $\tau(\omega_S)=\{\psi(1) | \psi\in\Hom_S(R^+,\omega_S)\}$ by Matlis duality applied to Theorem 5.1 in \cite{SmithTightClosureParameter},
	we have that $\tau(\omega_R, f^t)=\{\varphi(f^t) | \varphi\in \Hom(R^+,\omega_R)\}$ for $t\in\mathbb{Q}$.
	
	It follows that, for $t\in \Q$, $\tau(\omega_R, f^t)\supseteq\tau_{R^+}(\omega_R, f^t)$.
	Additionally, there exists $\epsilon>0$ such that $t+\epsilon\in\mathbb{Q}$ and
	$\tau(\omega_R, f^t)=\tau(\omega_R, f^{t+\epsilon})$.
	Since $\tau(\omega_R, f^{t+\epsilon})=\{\varphi(f^{t+\epsilon}) | \varphi\in \Hom(R^+,\omega_R)\}$,
	$\tau(\omega_R, f^t)\subseteq \tau_{R^+}(\omega_R, f^t)$.	
\end{proof}

We now come to our main result. 

\begin{proposition}[\textit{cf.} \cite{MustataTakagiWatanabeFThresholdsAndBernsteinSato, BlickleMustataSmithDiscretenessAndRationalityOfFThresholds}]\label{ThresholdsAndTest}
	Let $f$ be an element of a complete regular local ring $R$ and $B$ a big Cohen-Macaulay $R^+$-algebra. 
	\begin{enumerate}
		\item We have that $$\tau_{B}(f^{c^J_B(f)})\subseteq J.$$\label{ThresholdsAndTesta}
		\item For $\alpha\geq 0$, $$c_B^{\tau_{B}(f^{\alpha})}\left(f\right)\leq \alpha.$$\label{ThresholdsAndTestb}
	\end{enumerate}
\end{proposition}
\begin{proof}\,
	\begin{enumerate}
		\item Take $\alpha>c_B^J(f)$ such that $\tau_{B}(f^{c_B^J(f)})=\tau_{B}(f^{\alpha})$.
		Then, we have that $(fR^+)^{>\alpha}\subseteq JB$. It follows from the definition of $\tau_{B}(f^{\alpha})$ that
		$\tau_{B}(f^{\alpha})\subseteq J.$ Hence $\tau_{B}(f^{c_B^J(f)})\subseteq J.$
		
		\item By the definition of $\tau_{B}(f^{\alpha})$, we have that $(fR^+)^{>\alpha}\subseteq I_{B}(\tau_{B}(f^{\alpha}))$.
		By \autoref{lemmaEqual},  $(fR^+)^{>\alpha}\subseteq \tau_{B}(f^{\alpha})B$.
		It follows that $c^{\tau_{B}(I^{\alpha})}\left(f\right)\leq \alpha.$
	\end{enumerate}
\end{proof}

\begin{corollary}[\textit{cf.} \cite{MustataTakagiWatanabeFThresholdsAndBernsteinSato, BlickleMustataSmithDiscretenessAndRationalityOfFThresholds}]\label{Main}
	Let $f$ be an element of a complete regular local ring $R$ and $B$ a big Cohen-Macaulay $R^+$-algebra. The set of $B$-jumping numbers for $f$ is the same as the set of $B$-thresholds of $f$.
\end{corollary}
\begin{proof}
	Suppose that $\alpha$ is an $B$-jumping number.
	By \autoref{ThresholdsAndTest}\autoref{ThresholdsAndTestb}, $c_B^{\tau_{B}(f^{\alpha})}\left(f\right)\leq \alpha$ and, as a consequence,
	$$\tau_{B}(f^{\alpha})\subseteq \tau_{B}\left(f^{c_B^{\tau_{B}(f^{\alpha})}(f)}\right)$$
	By \autoref{ThresholdsAndTest}\autoref{ThresholdsAndTesta}, $$\tau_{B}\left(f^{c_B^{\tau_{B}(f^{\alpha})}(f)}\right)\subseteq \tau_{B}(f^{\alpha}).$$
	Hence, $$\tau_{B}(f^{\alpha})= \tau_{B}\left(f^{c_B^{\tau_{B}(f^{\alpha})}(f)}\right).$$
	Given that $\alpha$ is a jumping number and $c_B^{\tau_{B}(f^{\alpha})}\left(f\right)\leq \alpha$, $c_B^{\tau_{B}(I^{\alpha})}\left(I\right)=\alpha$.
	
	Let $c=c_B^J(f)$ for some $J$ and assume that $\tau_{B}(f^{c})= \tau_{B}(f^{c'})$  for some $c'<c$.
	By \autoref{ThresholdsAndTest}\autoref{ThresholdsAndTesta}, $\tau_{B}(f^{c'})=\tau_{B}(f^{c_B^J(f)})\subseteq J$ and so
	$\tau_{B}(f^{c'})\subseteq J.$
	It follows from the definition of $\tau_{B}(f^{c'})$ that $(fR^+)^{>c'}\subseteq I_{B}(\tau_{B}(f^{c'}))$.
	By \autoref{lemmaEqual},  $(fR^+)^{>c'}\subseteq \tau_{B}(f^{c'})B$.
	Thus, $(fR^+)^{>c'}\subseteq\tau_{B}(f^{c'})B\subseteq JB$ and $c'\geq c_B^J(f)$.
	It follows that $c>c'\geq c_B^J(f)$, a contradiction.
	Therefore, $c$ is a jumping number.
\end{proof}

\bibliographystyle{skalpha}
\bibliography{main}

\end{document}